\documentclass{amsart}
\usepackage{amssymb}
\usepackage{verbatim}
\usepackage{array}
\usepackage{latexsym}
\usepackage{enumerate}
\usepackage{amsmath}
\usepackage{amsfonts}
\usepackage{amsthm}
\usepackage{color}
\usepackage[english]{babel}

\newtheorem{theorem}{Theorem}[section]
\newtheorem{proposition}[theorem]{Proposition}
\newtheorem{lemma}[theorem]{Lemma}

\def\cQ{\mathcal Q}

\def\cX{\mathcal X}
\def\cY{\mathcal Y}

\def\K{\mathbb{K}}

\def\ord{\mbox{\rm ord}}
\def\deg{\mbox{\rm deg}}

\def\div{\mbox{\rm div}}

\def\Ker{\mbox{\rm Ker}}
\def\Div{\mbox{\rm div}}

\def\dim{\mbox{\rm dim}}
\def\supp{\mbox{\rm Supp}}

\def\gg{\mathfrak{g}}

\newcommand{\aut}{\mbox{\rm Aut}}

\def\supp{{\rm Supp}}

\newcommand{\ha}{{\textstyle\frac{1}{2}}}

\newcommand{\qa}{\textstyle\frac{1}{4}}

\title{Ordinary algebraic curves with many automorphisms in positive characteristic}
\date{}
\author{G\'abor Korchm\'aros and Maria Montanucci}
\begin{document}
\maketitle

\vspace{0.5cm}\noindent {\em Keywords}:
Algebraic curves, algebraic function fields, positive characteristic, automorphism groups.
\vspace{0.2cm}\noindent

\vspace{0.5cm}\noindent {\em Subject classifications}:
\vspace{0.2cm}\noindent  14H37, 14H05.


\begin{abstract} Let $\cX$ be an ordinary (projective, geometrically irreducible, nonsingular) algebraic curve of genus $\gg(\cX) \ge 2$ defined over an algebraically closed field $\K$ of odd characteristic $p$. Let $\aut(\cX)$ be the group of all automorphisms of $\cX$ which fix $\K$ element-wise.  For any solvable subgroup $G$ of $\aut(\cX)$ we prove that $|G|\leq 34 (\gg(\cX)+1)^{3/2}$. There are known curves attaining this bound up to the constant $34$. For $p$ odd, our result improves the classical Nakajima bound $|G|\leq 84(\gg(\cX)-1)\gg(\cX)$, see \cite{N}, and, for solvable groups $G$, the Gunby-Smith-Yuan bound $|G|\leq 6(\gg(\cX)^2+12\sqrt{21}\gg(\cX)^{3/2})$ where $\gg(\cX)>cp^2$ for some positive constant $c$; see \cite{GSY} .
\end{abstract}
    \section{Introduction}

 In this paper, $\cX$ stands for a (projective, geometrically irreducible, nonsingular) algebraic curve of genus $\gg(\cX) \ge 2$ defined over an algebraically closed field $\K$ of odd characteristic $p$. Let $\aut(\cX)$ be the group of all automorphisms of $\cX$ which fix $\K$ element-wise. The assumption $\gg(\cX)\geq 2$ ensures that $\aut(\cX)$ is finite. However the classical Hurwitz bound
 $|\aut(\cX)| \leq 84(\gg(\cX)-1)$ for complex curves fails in positive characteristic, and there exist four families of curves satisfying $|\aut(\cX)|\geq 8\gg^3(\cX)$; see \cite{stichtenoth1973II}, Henn \cite{henn1978}, and also \cite[Section 11.12]{HKT}. Each of them has has $p$-rank $\gamma(\cX)$ (equivalently, its Hasse-Witt invariant) equal to zero; see for instance \cite{GKLINZ}. On the other hand, if $\cX$ is ordinary, i.e. $\gg(\cX)=\gamma(\cX)$, Guralnik and Zieve announced in 2004, as reported in \cite{GSY,KR}, that for odd $p$  there exists a sharper bound, namely $|\aut(\cX)|\leq c_p\gg(\cX)^{8/5}$ with some constant $c_p$ depending on $p$. It should be noticed that no proof of this sharper bound is available in the literature. In this paper, we concern with solvable automorphism groups $G$ of an ordinary curve $\cX$, and for odd $p$ we prove the even sharper bound:
 \begin{theorem}
 \label{princ} Let $\cX$ be an algebraic curve of genus $\gg(\cX) \ge 2$ defined over an algebraically closed field $\K$ of odd characteristic $p$.
 If $\cX$ is ordinary and $G$ is a solvable subgroup of $\aut(\cX)$ then
 \begin{equation}
 \label{bound}
 |G|\leq 34(\gg(\cX)+1)^{3/2}.
 \end{equation}
 \end{theorem}
For odd $p$, our result provides an improvement on the classical Nakajima bound $|G|\leq 84(\gg(\cX)-1)\gg(\cX)$, see \cite{N}, and, for solvable groups, on the recent Gunby-Smith-Yuan bound $|G|\leq 6(\gg(\cX)^2+12\sqrt{21}\gg(\cX)^{3/2})$ proven in \cite{GSY} under the hypothesis that $\gg(\cX)>cp^2$ for some positive constant $c$.

The following example is due to Stichtenoth and it shows that (\ref{bound}) is the best possible bound apart from the constant $c$.
Let $\mathbb{F}_q$ be a finite field of order $q=p^h$ and let $\mathbb{K}=\overline{\mathbb{F}_q}$ denote its algebraic closure. For a positive integer $m$ prime to $p$, let $\cY$ be the irreducible curve with affine equation
\begin{equation}
\label{oct6eq1}
y^q+y=x^m+\frac{1}{x^m},
\end{equation}
and $F=\mathbb{K}(\cY))$ its function field. Let $t=x^{m(q-1)}$. The extension $F|\mathbb{K}(t)$ is a non-Galois extension as the Galois closure
of $F$ with respect to $H$ is the function field $\mathbb{K}(x,y,z)$ where $x,y,z$ are linked by (\ref{oct6eq1}) and $z^q+z=x^m$. Furthermore, $\gg(\cY)=(q-1)(qm-1)$, $\gamma(\cY)=(q-1)^2$ and  $\aut(\cY)$ contains a subgroup $Q\rtimes U$ of index $2$  where $Q$ is an elementary abelian normal subgroup of order $q^2$ and the complement $U$ is a cyclic group of order $m(q-1)$. If $m=1$ then $\cY$ is an ordinary curve, and in this case
$2\gg(\cY)^{3/2}=2(q-1)^3<2q^2(q-1)$ which shows indeed that (\ref{bound}) is sharp up to the constant $c$.


\section{Background and Preliminary Results}\label{sec2}
For a subgroup $G$ of $\aut(\cX)$, let $\bar \cX$ denote a non-singular model of $\K(\cX)^G$, that is,
a (projective non-singular geometrically irreducible) algebraic curve with function field $\K(\cX)^G$, where $\K(\cX)^G$ consists of all elements of $\K(\cX)$ fixed by every element in $G$. Usually, $\bar \cX$ is called the
quotient curve of $\cX$ by $G$ and denoted by $\cX/G$. The field extension $\K(\cX)|\K(\cX)^G$ is Galois of degree $|G|$.

Since our approach is mostly group theoretical, we prefer to use notation and terminology from Group theory rather than from Function field theory.

Let $\Phi$ be the cover of $\cX|\bar{\cX}$ where $\bar{\cX}=\cX/G$. A point $P\in\cX$ is a ramification point of $G$ if the stabilizer $G_P$ of $P$ in $G$ is nontrivial; the ramification index $e_P$ is $|G_P|$; a point $\bar{Q}\in\bar{\cX}$ is a branch point of $G$ if there is a ramification point $P\in \cX$ such that $\Phi(P)=\bar{Q}$; the ramification (branch) locus of $G$ is the set of all ramification (branch) points. The $G$-orbit of $P\in \cX$ is the subset
$o=\{R\mid R=g(P),\, g\in G\}$ of $\cX$, and it is {\em long} if $|o|=|G|$, otherwise $o$ is {\em short}. For a point $\bar{Q}$, the $G$-orbit $o$ lying over $\bar{Q}$ consists of all points $P\in\cX$ such that $\Phi(P)=\bar{Q}$. If $P\in o$ then $|o|=|G|/|G_P|$ and hence $\bar{\cQ}$ is a branch point if and only if $o$ is a short $G$-orbit. It may be that $G$ has no short orbits. This is the case if and only if every non-trivial element in $G$ is fixed--point-free on $\cX$, that is, the cover $\Phi$ is unramified. On the other hand, $G$ has a finite number of short orbits. For a non-negative integer $i$, the $i$-th ramification group of $\cX$
at $P$ is denoted by $G_P^{(i)}$ (or $G_i(P)$ as in \cite[Chapter
IV]{serre1979})  and defined to be
$$G_P^{(i)}=\{g\mid \ord_P(g(t)-t)\geq i+1, g\in
G_P\}, $$ where $t$ is a uniformizing element (local parameter) at
$P$. Here $G_P^{(0)}=G_P$.

Let $\bar{\gg}$ be the genus of the quotient curve $\bar{\cX}=\cX/G$. The Hurwitz
genus formula  gives the following equation
    \begin{equation}
    \label{eq1}
2\gg-2=|G|(2\bar{\gg}-2)+\sum_{P\in \cX} d_P.
    \end{equation}
    where the different $d_P$ at $P$ is given by
\begin{equation}
\label{eq1bis}
d_P= \sum_{i\geq 0}(|G_P^{(i)}|-1).
\end{equation}

Let $\gamma$ be the $p$-rank of $\cX$, and let $\bar{\gamma}$ be the $p$-rank of the quotient curve $\bar{\cX}=\cX/S$.
The Deuring-Shafarevich formula, see \cite{sullivan1975} or \cite[Theorem 11,62]{HKT}, states that
\begin{equation}
    \label{eq2deuring}
\gamma-1={|G|}(\bar{\gamma}-1)+\sum_{i=1}^k (|G|-\ell_i)
    \end{equation}
where $\ell_1,\ldots,\ell_k$ are the sizes of the short orbits of $G$.

A subgroup of $\aut(\cX)$ is  a prime to $p$ group (or a $p'$-subgroup) if its order is prime to $p$. A subgroup $G$ of $\aut(\cX)$ is {\em{tame}} if the $1$-point stabilizer of any point in $G$ is $p'$-group. Otherwise, $G$ is {\em{non-tame}} (or {\em{wild}}). Obviously, every $p'$-subgroup of $\aut(\cX)$ is tame, but the converse is not always true. By a theorem of Stichtenoth, see
\cite[Theorem 11.56]{HKT}, if $|G|>84(\gg(\cX)-1)$ then $G$ is non-tame.  An orbit $o$ of $G$ is {\em{tame}} if $G_P$ is a $p'$-group for $P\in o$. The stabilizer $G_P$ of a point $P\in \cX$ in $G$ is a semidirect product $G_P=Q_P\rtimes U$ where the normal subgroup $Q_P$ is a $p$-group while the complement $U$ is a cyclic prime to $p$ group; see \cite[Theorem 11.49]{HKT}. By a theorem of Serre, if $\cX$ is an ordinary curve then $Q_P$ is elementary abelian and no nontrivial element of $U$ commutes with a nontrivial element of $Q_P$; see \cite[Lemma 11.75]{HKT}. In particular, $|U|$ divides $|Q_P|-1$, see \cite[Proposition 1]{N}, and $G_P$ is not abelian when either $Q_P$ or $U$ is nontrivial.

The following two lemmas of independent interest play a role in our proof of Theorem \ref{princ}.
\begin{lemma}
\label{terrible} Let $\cX$ be an ordinary algebraic curve of genus $\gg(\cX) \ge 2$ defined over an algebraically closed field $\K$ of odd characteristic $p$. Let $H$ be a solvable automorphism group of $\aut(\cX)$ containing a normal $p$-subgroup $Q$ such that $|Q|$ and $[H:Q]$ are coprime. Suppose that a complement $U$ of $Q$ in $H$ is abelian and that
\begin{equation}
\label{eq22bisdic2015}
|H|> \left\{
\begin{array}{lll}
{\mbox{$18(\gg-1)$ for $|U|=3$}}, \\
{\mbox{$12(\gg-1)$ otherwise}}.
\end{array}
\right.
\end{equation}
Then $U$ is cyclic, and the quotient curve $\bar{\cX}=\cX/Q$ is rational. Furthermore, $Q$ has exactly two (non-tame) short orbits, say $\Omega_1, \Omega_2$. They are also the only short orbits of $H$, and $\gg(\cX)=|Q|-(|\Omega_1|+|\Omega_2|)+1.$
\end{lemma}
\begin{proof}
From the Schur-Zassenhaus theorem \cite[Corollary 7.5]{MA}, $H=Q\rtimes U$. Set $|Q|=p^k$, $|U|=u$. Then $p\nmid u$. Furthermore, if $u=2$ then $|H|=2|Q|>9\gg(\cX)$ whence $|Q|>4.5\gg(\cX)$. From Nakajima's bound $\cX$, see \cite[Theorem 11.84]{HKT}, has zero $p$-rank. Therefore $u\geq 3$.

Three cases are treated separately according as the quotient curve $\bar{\cX}=\cX/Q$ has genus $\bar{\gg}$ at least $2$, or $\bar{\cX}$ is elliptic, or rational.

If $\gg(\bar{\cX})\geq 2$, then $\aut(\bar{\cX})$ has a subgroup isomorphic to $U$, and \cite[Theorem 11.79]{HKT} yields $4\gg(\bar{\cX})+4\geq |U|.$ Furthermore, from
the Hurwitz genus formula applied to $Q$, $\gg-1\geq |Q|(\gg(\bar{\cX})-1)$. Therefore, if $c=12$ or $c=18$, according as $|U|>3$ or $|U|=3$,
$$(4\gg(\bar{\cX})+4)|Q|\geq |U||Q|=|H|> c(\gg-1)= 12 |Q|(\gg(\bar{\cX})-1),$$
whence $$c< 4\frac{\gg(\bar{\cX})+1}{\gg(\bar{\cX})-1} \leq 12,$$
a contradiction.

If $\bar{\cX}$ is elliptic, then the cover $\cX|\bar{\cX}$ ramifies, otherwise $\cX$ itself would be elliptic. Thus, $Q$ has some short orbits. Take one of them together with its images
$o_1,\ldots,o_{u_1}$ under the action of $H$. Since $Q$ is a normal subgroup of $H$, $o=o_1\cup\ldots\cup o_{u_1}$ is a $H$-orbit of size $u_1p^v$ where $p^v=|o_1|=\ldots |o_{u_1}|$. Equivalently, the stabilizer of a point $P\in o$ has order
$p^{k-v}u/u_1$, and it is the semidirect product $Q_1\rtimes U_1$ where $|Q_1|=p^{k-v}$ and $|U_1|=u/u_1$ for a subgroup $Q_1$ of $Q$ and $U_1$ of $U$ respectively.
The point $\bar{P}$ lying under $P$ in the cover $\cX| \bar{\cX}$ is fixed by the factor group $\bar{U_1}=U_1Q/Q$. Since $\bar{\cX}$ is elliptic, \cite[Theorem 11.94]{HKT} implies $|\bar{U}_1|\leq 12$ for $p=3$ and $|\bar{U}_1|\leq 6$ for $p>3$.
As $\bar{U}_1\cong U_1$, this yields the same bound for $|U_1|$, that is, $u\leq 4 u_1$ for $p=3$ and $u\leq 6 u_1$ for $p>3$. 
Furthermore, $d_P\geq 2(p^{k-v}-1)\geq \textstyle\frac{4}{3} p^{k-v}$. From the Hurwitz genus formula applied to $Q$, if $p=3$ then
$$2\gg-2\geq 3^v u_1d_P\geq 3^v u_1(\textstyle\frac{4}{3} 3^{k-v}) \geq \textstyle\frac{4}{3} 3^k u_1\geq \textstyle\frac{1}{3} 3^k u=\textstyle\frac{1}{3}|Q||U|=\textstyle\frac{1}{3} |H|,$$
while for $p>3$,
$$2\gg-2\geq p^v u_1d_P\geq p^v u_1(\textstyle\frac{4}{3} p^{k-v}) \geq \textstyle\frac{4}{3} p^k u_1\geq \textstyle\frac{2}{9} p^ku=\textstyle\frac{2}{9}|Q||U|=\textstyle\frac{2}{9} |H|,$$
But this contradicts (\ref{eq22bisdic2015}).

If $\bar{\cX}$ is rational, then $Q$ has at least one short orbit. Furthermore, $\bar{U}=UQ/Q$ is isomorphic to a subgroup of $PGL(2,\K)\cong\aut(\bar{\cX})$. Since $U\cong \bar{U}$,  the classification of finite subgroups of $PGL(2,\K)$, see \cite{maddenevalentini1982}, shows that $U$ is cyclic, $\bar{U}$ fixes two points $\bar{P}_0$ and $\bar{P}_\infty$ but no nontrivial element in $\bar{U}$ fixes a point other than $\bar{P}_0$ and $\bar{P}_\infty$. Let $o_\infty$ and $o_0$ be the $Q$-orbits lying over $\bar{P}_0$ and $\bar{P}_\infty$, respectively. Obviously, $o_\infty$ and $o_0$ are short orbits of $H$. We show that $Q$ has at most two short orbits, the candidates being $o_\infty$ and $o_0$. By absurd, there is a $Q$-orbit $o$ of size $p^m$ with $m<k$ which lies over a point $\bar{P}\in \bar{\cX}$ different from both $\bar{P}_0$ and $\bar{P}_\infty$. Since the orbit of $\bar{P}$ in $\bar{U}$ has length $u$, then the $H$-orbit of a point $P\in o$ has length $up^m$. If $u>3$, the Hurwitz genus formula applied to $Q$ gives $$2\gg-2\ge -2p^k+up^m(p^{k-m}-1)\geq -2p^k+up^m\textstyle\frac{2}{3} p^{k-m}\geq -2p^k+\textstyle\frac{2}{3}up^k\geq
\textstyle\frac{2}{3}(u-3)p^k \geq \textstyle\frac{1}{6}up^k = \frac{1}{6}|H|,$$
a contradiction with $|H|> 12(\gg-1)$. If $u=3$ then $p>3$, and hence $$2\gg-2\ge -2p^k+3p^m(p^{k-m}-1)=p^k-3p^m>\textstyle\frac{1}{3} p^k,$$
whence $|H|=3p^k<18(\gg-1)$, a contradiction with (\ref{eq22bisdic2015}).
This proves that $H$ has exactly two short orbits. Assume that $Q$ has two short orbits. They are $o_\infty$ and $o_0$. If their lengths are $p^a$ and $p^b$ with $a,b<k$, the Deuring-Shafarevich formula applied to $Q$ gives
$$\gg(\cX)-1=\gamma(\cX)-1=-p^k+(p^k-p^a)+p^k-p^b$$
whence $\gg(\cX)=p^k-(p^a+p^b)+1>0$. The same argument shows that if $Q$ has just one short orbit, then $\gamma(\cX)=0$, a contradiction.
\end{proof}
\begin{lemma}  \label{GK61}
 Let $N$ be an automorphism group of an algebraic curve of even genus such that $|N|$ is even. Then any $2$-subgroup of $N$ has a cyclic subgroup of index $2$.
\end{lemma}
\begin{proof} Let $U$ be a subgroup of $\aut(\cX)$ of order $d=2^u\geq 2$, and $\bar{\cX}=\cX/N$ the arising quotient curve. From  the Hurwitz genus formula applied to $U$,
$$2\gg(\cX)-2=2(p-2)p^{n-1}=2^u(2\gg\bar{\cX}-2)+\sum_{i=1}^m(2^u-\ell_i)$$
where $\ell_1,\ldots,\ell_m$ are the short orbits of $U$ on $\cX$. Since $2(p-2)p^{n-1}\equiv 2 \pmod 4$ while
$2^u(2\gg \bar{\cX}-2)\equiv 0 \pmod 4$, some $\ell_i$ ($1\le i \le m$) must be either $1$ or $2$.  Therefore, $U$ or a subgroup of $U$ of index $2$ fixes a point of $\cX$ and hence
is cyclic.
\end{proof}
\section{The proof of Theorem \ref{princ}} The assertion holds for $\gg(\cX)=2$, as $|G|\leq 48$ for any solvable automorphism group $G$ of a genus two curve; see \cite[Proposition 11.99]{HKT}.
For $\gg(\cX)>2$, $\cX$ is taken by absurd for a minimal counterexample with respect the genera so that for any solvable subgroup of $\aut(\bar{\cX})$ of an ordinary curve $\bar{\cX}$ of genus $\gg(\bar{\cX})\ge 2$ we have $|\bar{G}|\leq 34(\gg+1)^{3/2}$. Two cases are treated separately.
\subsection{Case I: $G$ contains a minimal normal $p$-subgroup}
\label{CasoI}
\begin{proposition} \label{caseI}  Let $\cX$ be an ordinary algebraic curve of genus $\gg$ defined over an algebraically closed field $\mathbb{K}$ of odd characteristic $p>0$. If $G$ is a solvable subgroup of $\aut(\cX)$ containing a minimal normal $p$-subgroup $N$, then $|G|\leq 34 (\gg+1)^{3/2}$.
\end{proposition}
\begin{proof}
Take the largest normal $p$-subgroup $Q$ of $G$. Let $\bar{\cX}$ be the quotient curve of $\cX$ with respect to $Q$ and let $\bar{G}=G/Q$. The quotient group $\bar{G}$ is a subgroup of $\aut(\bar{\cX})$ and it has no normal $p$-subgroup, otherwise $G$ would have a normal $p$-subgroup properly containing $Q$. For $\bar{\gg}=\gg(\bar{\cX})$ three cases may occur, namely $\bar{\gg} \geq 2$, $\bar{\gg}=1$ or $\bar{\gg}=0$. If $\bar{\gg} \geq 2$, from the Hurwitz genus formula,
$$2\gg-2 \geq |Q|(2\bar{\gg}-2)=\frac{|G|}{|\bar{G}|}(2\bar{\gg}-2)$$
whence $|\bar{G}|> 34(\bar{\gg}+1)^{3/2}$. Since $\bar{\cX}$ is still ordinary, this contradicts our choice of $\cX$ to be a minimal counterexample.
If $\bar{\gg}=1$ then the cover $\mathbb{K}(\cX)| \mathbb{K}(\bar{\cX})$ ramifies. Take short orbit $\Delta$ of $Q$. Let $\Gamma$ be the non-tame short orbit of $G$ that contains $\Delta$. Since $Q$ is normal in $G$, the orbit $\Gamma$ partitions into short orbits of $Q$ whose components have the same length which is equal to $|\Delta|$. Let $k$ be the number of the $Q$-orbits contained in $\Gamma$. Then,
$$|G_P|= \frac{|G|}{k|\Delta|},$$
holds for every $P \in \Gamma$. Moreover, the quotient group $G_PQ/Q$ fixes a place on $\bar{\cX}$. Now, from \cite[Theorem 11.94 (ii)]{HKT},
$$\frac{|G_PQ|}{|Q|} = \frac{|G_P|}{|G_P \cap Q|} = \frac{|G_P|}{|Q_P|}\leq 12.$$
From this together with the Hurwitz genus formula and  \cite[Theorem 2 (i)]{N},
$$2\gg-2 \geq 2k|\Delta|(|Q_P|-1)\geq 2k|\Delta|\frac{|Q_P|}{2} \geq \frac{k|\Delta||G_P|}{12}=\frac{|G|}{12},$$
which contradicts our hypothesis $|G| >34(\gg+1)^{3/2}$.

It turns out that $\bar{\cX}$ is rational. Therefore $\bar{G}$ is isomorphic to a subgroup of $PGL(2,\mathbb{K})$ which contains no normal $p$-subgroup. From the classification of finite subgroups of $PGL(2,\K)$, see \cite{maddenevalentini1982}, $\bar{G}$ is a prime to $p$ subgroup which is either cyclic, or dihedral, or isomorphic to one of the the groups $\rm{Alt}_4,\rm{Sym}_4$. In all cases, $\bar{G}$ has a cyclic subgroup $U$ of index $\le 6$ and of order distinct from $3$. We may dismiss all cases but the cyclic one assuming that $G=Q\rtimes U$ with $|G|\geq \textstyle\frac{5}{3}(\gg(\cX)+1)^{3/2}.$ Then $|G|> 12(\gg-1)$. Therefore, Lemma \ref{terrible} applies to $G$. Thus, $Q$ has exactly two (non-tame) orbits, say $\Omega_1$ and $\Omega_2$, and they are also the only short orbits of $G$. More precisely,
\begin{equation}
\label{eq1zi} \gamma-1=|Q|-(|\Omega_1|+|\Omega_2|).
\end{equation}
We may also observe that $G_P$ with $P\in \Omega_1$ contains a subgroup $V$ isomorphic to $U$. In fact, $|Q||U|=|G|=|G_P||\Omega_1|=|Q_P\rtimes V||\Omega_1|=|V||Q_P||\Omega_1|$ with a prime to $p$ subgroup $V$ fixing $P$, whence $|U|=|V|$. Since $V$ is cyclic the claims follows.

We go on with the case where both $\Omega_1$ and $\Omega_2$ are nontrivial, that is, their lengths are at least $2$.

Assume that $Q$ is non-abelian and look at the action of its center $Z(Q)$ on $\cX$. Since $Z(Q)$ is a nontrivial normal subgroup of $G$, arguing as before we get that the factor group $\cX/Z(Q)$ is rational, and hence the Galois cover $\cX|(\cX/Z(Q))$ ramifies at some points. In other words, there is a point $P\in \Omega_1$ (or $R\in \Omega_2$) such that some nontrivial subgroup $T$ of $Z(Q)$ fixes $P$ (or $Q$). Suppose that the former case occurs. Since $\Omega_1$ is a $Q$-orbit, $T$ fixes $\Omega_1$ pointwise.

The group $G$ induces a permutation group on $\Omega_1$ and let $M_1$ be the kernel of this permutation representation. Obviously, $T$ is a nontrivial $p$-subgroup of $M_1$. Therefore
$M$ contains some but not all elements from $Q$. Since both $M_1$ and $Q$ are normal subgroups of $G$,  $N=M_1\cap Q$ is a nontrivial normal $p$-subgroup of $G$. From (ii), the quotient curve $\tilde{\cX}=
\cX/N$ is rational, and hence the factor group $\tilde{G}=G/N$ is isomorphic to a subgroup of $PGL(2,\mathbb{K})$. Since $1\lvertneqq N\lvertneqq Q$, the order of $\tilde{G}$ is divisible by $p$. From the classification of subgroups of $PGL(2,\mathbb{K})$, see \cite{maddenevalentini1982}, $\tilde{G}=\tilde{Q}\rtimes \tilde{U}$ where
$\tilde{Q}$ is an elementary abelian $p$-group of order $q$ and $\tilde{U}\cong UN/N\cong U$ with $|\tilde{U}|=|U|$ is a divisor of $q-1$.

This shows that $Q$ acts on $\Omega_1$ as an abelian transitive permutation group. Obviously this holds true when $Q$ is abelian. Therefore, the action of $Q$ on $\Omega_1$ is sharply transitive. In terms of $1$-point stabilizers of $Q$ on $\Omega_1$, we have $Q_P=Q_{P'}$ for any $P,P'\in \Omega_1$. Moreover, $Q_P=N$, and hence $Q_P$ is a normal subgroup of $G$.

Furthermore, since $\cX$ is an ordinary curve, $Q_P$ is an elementary abelian group by \cite[Theorem 2 (i)]{N} and \cite[Theorem 11.74 (iii)]{HKT}.

The quotient curve $\cX/Q_P$ is rational and its automorphism group contains the factor group $Q/Q_P$. Hence, exactly one of the $Q_P$-orbits is preserved by $Q$. Since $\Omega_1$ is a $Q$-orbit consisting of fixed points of $Q_P$, $\Omega_2$ must be a $Q_P$-orbit. Similarly, if $Z(Q)\neq Q_P$, the factor group $Z(Q)Q_P/Q_P$ is an automorphism group of $\cX/Q_P$ and hence exactly one of the $Q_P$-orbits is preserved by $Z(Q)$. Either $Z(Q)$ fixes a point in $\Omega_1$ but then
$Z(Q)=Q_P$, or $\Omega_2$ is a $Z(Q)$-orbit. This shows that either $Z(Q)=Q_P$ or $Z(G)$ acts transitively on $\Omega_2$.

Two cases arise according as $Q_P$ is sharply transitive and faithful on $\Omega_2$ or some nontrivial element in $Q_P$ fixes $\Omega_2$ pointwise.

If some nontrivial element in $Q_P$ fixes $\Omega_2$ pointwise then the kernel $M_2$ of the permutation representation of $H$ on $\Omega_2$ contains a nontrivial $p$-subgroup. Hence the above results extends from $\Omega_1$ to $\Omega_2$, and $Q_R$ is a normal subgroup of $Q$.

If $Q_P$ is (sharply) transitive on $\Omega_2$ then the abelian group $Z(Q)Q_P$ acts on $\Omega_2$ as a sharply transitive permutation group, as well. Hence either $Z(Q)=Q_P$, or as before $M_2$ contains a nontrivial $p$-subgroup, and $Q_R$ is a normal subgroup of $Q$. In the former case, $Q=Q_PQ_R$ and $Q_R\cap Q_P=\{1\}$, and $Z(Q)=Q_P$ yields that
\begin{equation}
\label{eq4zi}
Q=Q_P\times Q_R.
\end{equation}
This shows that $Q$ is abelian, and hence $|Q|\leq 4\gg+4$. Also, either $|Q_P|$ (or $|Q_R|$) is at most $\sqrt{4\gg+4}$. Since $\cX$ is an ordinary curve, the second ramification group $G_P^{(2)}$ at $P\in \Omega_1$ is trivial. For $G_P=Q_P\rtimes V$, \cite[Lemma 11.81]{HKT} gives  $|U|=|V|\leq |Q_P|-1$. Hence
$|U|<|Q_P|\leq \sqrt{|Q|}\leq  \sqrt{4\gg+4}$ whence
\begin{equation}
\label{eq5zi}
|G|=|U||Q|\leq 8 (\gg+1)^{3/2}.
\end{equation}
If $Q_R$ is a normal subgroup, take a point $R$ from $\Omega_2$, and look at the subgroup $Q_{P,R}$ of $Q_P$ fixing $R$. Actually, we prove that
that either $Q_{P,R}=Q_P$ or $Q_{P,R}$ is trivial. Suppose that  $Q_{P,R}\neq \{1\}$. Since $Q_{P,R}=Q_P\cap Q_R$ and both $Q_P$ and $Q_R$ are normal subgroups of $G$, them same holds for $Q_{P,R}$. By (ii), the quotient curve $\cX/Q_{P,R}$ is rational and hence its automorphism group $Q/Q_{P,R}$ fixes exactly one point. Furthermore, each point in $\Omega_2$ is totally ramified.
Therefore, $Q_R=Q_{P,R}$, otherwise $Q_R/Q_{P,R}$ would fix any point lying under a point in $\Omega_1$ in the cover $\cX|(\cX/Q_{P,R})$.

It turns out that either $Q_P=Q_R$ or $Q_P\cap Q_R=\{1\}$, whenever $P\in \Omega_1$ and $R\in \Omega_2$.

In the former case, from the Deuring-Shafarevi\v{c} formula applied to $Q_P$,
$$\gamma-1=-|Q_P|+|\Omega_1|(|Q_P|-1)+|\Omega_2|(|Q_{P}|-1)=-|Q_P|+|Q|-|\Omega_1|+|Q|-|\Omega_2|.$$
This together with (\ref{eq1zi}) give $Q=Q_P$, a contradiction.

Therefore, the latter case must hold. Thus $Q=Q_P\times Q_R$ and $Q_P$ (and also $Q_R$) is an elementary abelian group since is isomorphic to a $p$-subgroup of $PGL(2,\mathbb{K})$. Also,
$|Q_P|=|Q_R|=\sqrt{|Q|}$. Since $Q$ is abelian, this yields $|Q_P|\leq \sqrt{4\gg+4}$.
Now, the argument used after (\ref{eq4zi}) can be employed to prove (\ref{eq5zi}). This ends the proof in the case where both $\Omega_1$ and $\Omega_2$ are nontrivial.

Suppose next $\Omega_1=\{P\}$ and $|\Omega_2| \geq 2$. Then $G$ fixes $P$, and hence $G=Q \rtimes U$ with an elementary abelian $p$-group $Q$.
Furthermore, $G$ has a permutation representation on $\Omega_2$ with kernel $K$. As $\Omega_2$ is a short orbit of $Q$, the stabilizer $Q_R$ of $R\in \Omega_2$ in $Q$ is nontrivial. Since $Q$ is abelian, this yields that $K$ is nontrivial, and hence it is a nontrivial elementary abelian normal subgroup of $G$. In other words, $Q$ is an $r$-dimensional vector space $V(r,p)$  over a finite field $\mathbb{F}_p$ with $|Q|=p^r$, the action of each nontrivial element of $U$ by conjugacy is a nontrivial automorphism of $V(r,p)$, and $K$ is a $K$-invariant subspace. By Maschke's theorem, see for instance \cite[Theorem 6.1]{MA}, $K$ has a complementary $U$-invariant subspace. Therefore, $Q$ has a subgroup $M$ such that $Q=K\times M$, and $M$ is a normal subgroup of $G$. Since $K\cap M=\{1\}$, and $\Omega_2$ is an orbit of $Q$, this yields $|M|=|\Omega_2|$. The factor group $G/M$ is an automorphism group of the quotient curve $\cX/M$, and $Q/M$ is a nontrivial $p$-subgroup of $G/M$ whereas $G/M$ fixes two points on $\cX/M$. Therefore the quotient curve $\cX/M$ is not rational since the $2$-point stabilizer in the representation of $PGL(2,\mathbb{K})$ as an automorphism group of the rational function field is a prime to $p$ (cyclic) group. We show that $\cX/M$ is neither elliptic. From the Deuring-Shafarevi\v{c} formula. $\gg(\cX)-1=\gamma(\cX)-1= -|Q| +1 + |\Omega_2|$, and so $\gg(\cX)$ is even. Since $M$ is a normal subgroup of odd order, $\gg(\cX)\equiv 0 \pmod 2$ yields that $\gg(\cX/M)\equiv 0 \pmod 2$. In particular, $\gg(\cX/M) \ne 1$. Therefore, $\gg(\cX/M) \geq 2$. At this point we may repeat our previous argument and prove $|G/M|>34(\gg(\cX/M)+1)^{3/2}$. Again, a contradiction to our choice of $\cX$ to be a minimal counterexample, which ends the proof in the case where just one of $\Omega_1$ and $\Omega_2$ is trivial.

We are left with the case where both short orbits of $Q$ are trivial. Our goal is to prove a much stronger bound for this case, namely $|U|\leq 2$ whence
\begin{equation}
\label{eq10oct2016} |G|\leq 2(\gg(\cX)+1).
\end{equation}
We also show that if equality holds then $\cX$ is a hyperelliptic curve with equation
\begin{equation}
\label{eq010oct2016} f(U)=aT+b+cT^{-1},\,\,a,b,c\in \mathbb{K}^*,
\end{equation}
where $f(U)\in \mathbb{K}[U]$ is an additive polynomial of degree $|Q|$.

Let $\Omega_1=\{P_1\}$ and $\Omega_2=\{P_2\}$. Then $Q$ has two fixed points $P_1$ and $P_2$ but no nontrivial element in $Q$ fixes a point of $\cX$ other than $P_1$ and $P_2$.
From the Deuring-Shafarevi\v{c} formula
\begin{equation} \label{eq4n}
\gg(\cX)+1=\gamma(\cX)+1=|Q|.
\end{equation}
Therefore, $|U|\leq \gg(\cX)$.
Actually, for our purpose, we need a stronger estimate, namely $|U|\le 2$. To prove the latter bound, we use some ideas from Nakajima's paper \cite{N}.
regarding the Riemann-Roch spaces $\mathcal{L}(\mathbf{D})$ of certain divisors $\mathbf{D}$ of $\mathbb{K}(\cX)$. Our first step is to show
\begin{enumerate}
\item[(i)] $\dim_{\mathbb{K}} \mathcal{L}((|Q|-1)P_1)=1$,
\item[(ii)] $\dim_{\mathbb{K}} \mathcal{L}((|Q|-1)P_1+P_2)\geq 2$.
\end{enumerate}
Let $\ell \geq 1$ be the smallest integer such that $\dim_{\mathbb{K}} \mathcal{L}(\ell P_1)=2$, and take $x \in \mathcal{L}(\ell P_1)$ with $v_{P_1}(x)=-\ell$. As $Q=Q_{P_1}$, the Riemann-Roch space $\mathcal{L}(\ell P_1)$ contains all $c_\sigma=\sigma(x)-x$ with $\sigma\in Q$. This yields $c_\sigma\in \mathbb{K}$ by $v_{P_1}(c_\sigma) \geq -\ell+1$ and our choice of $\ell$ to be minimal. Also, $Q=Q_{P_2}$ together with $v_{P_2}(x) \geq 0$ show $v_{P_2}(c_\sigma) \geq 1$. Therefore $c_\sigma =0$ for all $\sigma \in Q$, that is, $x$ is fixed by $Q$. From
$\ell=[\mathbb{K}(\cX):\mathbb{K}(x)]=[\mathbb{K}:\mathbb{K}(\cX)^Q][\mathbb{K(\cX)}^Q:\mathbb{K}(x)]$ and $|Q|=[\mathbb{K}:\mathbb{K}(\cX)^Q]$, it turns out that $\ell$ is a multiple of $|Q|$. Thus $\ell>|Q|-1$ whence (i) follows. From the Riemann-Roch theorem, $\dim_{\mathbb{K}}\mathcal{L}((|Q|-1)P_1+P_2)\geq |Q|-g+1=2$ which proves (ii).

Let $d \geq 1$ be the smallest integer such that $\dim_{\mathbb{K}} \mathcal{L}(dP_1+P_2)=2$. From (ii)
\begin{equation} \label{in1}
 d \leq |Q|-1.
\end{equation}
Let $\alpha$ be a generator of the cyclic group $U$. Since $\alpha$ fixes both points $P_1$ and $P_2$, it acts on $\mathcal{L}(dP_1+P_2)$ as a $\mathbb{K}$-vector space automorphism $\bar{\alpha}$. If $\bar{\alpha}$ is trivial then $\alpha(u)=u$ for all $u\in\mathcal{L}(dP_1+P_2)$. Suppose that $\bar{\alpha}$ is nontrivial. Since $U$ is a prime to $p$ cyclic group, $\bar{\alpha}$ has two distinct eigenspaces, so that $\mathcal{L}(dP_1+P_2)=\mathbb{K}\oplus \mathbb{K}u$ where $u\in \mathcal{L}(dP_1+P_2)$ is an eigenvector of $\bar{\alpha}$ with eigenvalue $\xi\in \mathbb{K}^*$ so that $\bar{\alpha}(u)=\xi u$ with $\xi^{|U|}=1$. Therefore there is $u\in \mathcal{L}(dP_1+P_2)$ with $u\neq 0$ such that $\alpha(u)=\xi u$ with $\xi^{|U|}=1$.  The pole divisor
of $u$ is
\begin{equation} \label{equ1}
\Div(u)_\infty = dP_1+P_2.
\end{equation}
Since $Q=Q_{P_1}=Q_{P_2}$, the Riemann-Roch space $\mathcal{L}(dP_1+P_2)$ contains $\sigma(u)$ and hence it contains all
$$\theta_\sigma=\sigma(u)-u,\,\, \sigma\in Q.$$
By our choice of $d$ to be minimal, this yields $\theta_\sigma\in \mathbb{K}$, and then defines the map $\theta$ from $Q$ into $\mathbb{K}$ that takes $\sigma$ to $\theta_\sigma$. More precisely, $\theta$ is
a homomorphism from $Q$ into the additive group $(\mathbb{K},+)$ of $\mathbb{K}$ as the following computation shows:
$$\theta_{\sigma_1\circ\sigma_2}=(\sigma_1\circ\sigma_2)(u)-u=\sigma_1(\sigma_2(u)-u+u)-u=\sigma_1(\theta_{\sigma_2})+\sigma_1(u)-u=\theta_{\sigma_2}+\theta_{\sigma_1}=\theta_{\sigma_1}+\theta_{\sigma_2}.$$  Also, $\theta$ is injective. In fact, if $\theta_{\sigma_0}=0$ for some $\sigma_0 \in Q \setminus \{1\}$, then $u$ is in the fixed field of $\sigma_0$, which is impossible since $v_{P_2}(u)=-1$ whereas $P_2$ is totally ramified in the cover $\cX|(\cX/\langle \sigma_p \rangle)$. The image $\theta(Q)$ of $\theta$ is an additive subgroup of $\mathbb{K}$ of order $|Q|$. The smallest subfield of $\mathbb{K}$ containing $\theta(Q)$ is a finite field $\mathbb{F}_{p^m}$ and hence $\theta(Q)$ can be viewed as a linear subspace of $\mathbb{F}_{p^m}$  considered as a vector space over $\mathbb{F}_p$.  Therefore the polynomial
\begin{equation}
\label{eq210oct2016}
f(U)=\prod_{\sigma \in Q} (U-\theta_\sigma)
\end{equation}
is a linearized polynomial over $\mathbb{F}_p$; see \cite[Section 4, Theorem 3.52]{LN}. In particular, $f(U)$ is an additive polynomial of degree $|Q|$; see also \cite[Chapter V, $\mathsection \ 5$]{Serre}. Also, $f(U)$ is separable as $\theta$ is injective. From (\ref{eq210oct2016}), the pole divisor of $f(u)\in \mathbb{K}(\cX)$ is
\begin{equation}
\label{eq310oct2016}
\div(f(u))_\infty=|Q|(dP_1+P_2).
\end{equation}
For every $\sigma_0\in Q$,
$$\sigma_0(f(u))=\prod_{\sigma\in Q} (\sigma_0(u)-\theta_{\sigma})=\prod_{\sigma\in Q}(u+\theta_{\sigma_0}-\theta_{\sigma})=\prod_{\sigma\in Q}(u-\theta_{\sigma\sigma_0^{-1}})=\prod_{\sigma\in Q}(u-\theta_{\sigma})=f(u).$$
Thus $f(u) \in \mathbb{K}(\cX)^Q$. Furthermore, from $\alpha\in N_G(Q)$, for every $\sigma\in Q$ there is
$\sigma'\in Q$ such that $\alpha\sigma=\sigma'\alpha$. Therefore
$$\alpha(f(u))=\prod_{\sigma\in Q}(\alpha(\sigma(u)-u))=\prod_{\sigma\in Q}(\alpha(\sigma(u))-\xi u)=\prod_{\sigma\in Q}(\sigma'(\alpha(u))-\xi u)=\prod_{\sigma\in Q}(\sigma'(\xi u)-\xi u)=\xi f(u).$$
This shows that if $R\in\cX$ is a zero of $f(u)$ then $\supp(\div(f(u)_0))$ contains the $U$-orbit of $R$ of length $|U|$. Actually, since $\sigma(f(u))=f(u)$ for $\sigma\in Q$, $\supp(\div(f(u)_0))$ contains the $G$-orbit of $R$ of length $|G|=|Q||U|$. This together with (\ref{eq310oct2016}) give
\begin{equation}
\label{eq410oct2016}
|U|\big|(d+1).
\end{equation}
On the other hand, $\mathbb{K}(\cX)^Q$ is rational. Let $\bar{P_1}$ and $\bar{P_2}$ the points lying under $P_1$ and $P_2$, respectively, and let $\bar{R_1},\bar{R_2},\ldots \bar{R_k}$ with $k=(d+1)/|U|$ be the points lying under the zeros of $f(u)$ in the cover $\cX|(\cX/Q)$. We may represent $\mathbb{K}(\cX)^Q$ as the projective line $\mathbb{K}\cup \{\infty\}$ over $\mathbb{K}$ so that $\bar{P_1}=\infty$, $\bar{P_1}=0$ and $\bar{R}_i=t_i$ for $1\leq i \leq k$. Let $g(t)=t^d+t^{-1}+h(t)$
where $h(t)\in \mathbb{K}[t]$ is a polynomial of degree $k=(d+1)/|U|$ whose roots are $r_1,\ldots,r_k$. It turns out that $f(u),g(t)\in \mathbb{K}(\cX)$ have the same pole and zero divisors, and hence
\begin{equation}
\label{eq510oct2016} cf(u)=t^d+t^{-1}+h(t),\,\ c\in \mathbb{K}^*.
\end{equation}
We prove that $\mathbb{K}(\cX)=\mathbb{K}(u,t)$. From \cite{sullivan1975}, see also \cite[Remark 12.12]{HKT}, the polynomial $cTf(X)-T^{d+1}-1-h(T)T$ is irreducible, and the plane curve $\mathcal{C}$ has genus $\gg(\mathcal{C})=\ha (q-1)(d+1)$. Comparison with (\ref{eq4n}) shows $\mathbb{K}(\cX)=\mathbb{\mathcal{C}}$ and $d=1$ whence $|U|\leq 2$. If equality holds then
$\deg\,h(T)=1$ and $\cX$ is a hyperelliptic curve
with equation (\ref{eq010oct2016}).
\end{proof}

\subsection{Case II: $G$ contains no minimal normal $p$-subgroup}
\label{CasoII}
\begin{proposition} \label{caseII}  Let $\cX$ be an ordinary algebraic curve of genus $\gg$ defined over a field $\mathbb{K}$ of odd characteristic $p>0$. If $G$ is a solvable subgroup of $\aut(\cX)$ with a minimal normal subgroup $N$ satisfying Case II, then $|G|\leq 34 (\gg(\cX)+1)^{3/2}$.
\end{proposition}
\begin{proof}
From Proposition \ref{caseI}, $G$ contains no nontrivial normal $p$-subgroup.
The factor group $\bar{G}=G/N$ is a subgroup of $\aut(\bar{\cX})$ where $\bar{\cX}=\cX/N$. Furthermore, $|N|\leq 4\gg(\cX)+4$ as $N$ is abelian; see \cite[Theorem 11.79]{HKT}.
 Arguing as in Section \ref{CasoI}, we see that $\bar{\cX}$ is either rational or elliptic, that is, $\gg(\cX)\ge 2$ is impossible.

 A detailed description of the arguments is given below. Assume that $\cX$ is a minimal counterexample with respect to the genus. Since $\cX$ is ordinary, any $p$-subgroup $S$ of $G$ is an elementary abelian group and it has a trivial second ramification group at any point $\cX$. The latter property remains true when $\cX$ is replaced by $\bar{\cX}$. To show this claim, take $\bar{P}\in \bar{\cX}$ and let $\bar{S}_{\bar{P}}$ the subgroup of the factor group $\bar{S}=SN/N$ fixing $\bar{P}$. Since $p\nmid |N|$ there is a point $P\in \cX$ lying over
 $\bar{P}$ which is fixed by $S$. Hence the stabilizer $S_P$ of $P$ in $S$ is a nontrivial normal subgroup of $G_P$. Since $N$ is a normal subgroup in $G$, so is $N_P$ in $G_P$. This yields that
 the product $N_PS_P$ is actually a direct product. Therefore $N_P$ is trivial, otherwise the second ramification group of $S$ in $P$ is nontrivial by a result of Serre; see \cite[Theorem 11.75]{HKT}.
Actually, $N$ may be assumed to be the largest normal subgroup $N_1$ of $G$ whose order is prime to $p$. If the quotient curve $\cX_1=\cX/N_1$ is neither rational, nor elliptic then its $\mathbb{K}$-automorphism group $G_1=G/N_1$ has order bigger than $34(\gg(\cX_1)+1)^{3/2}$, by the Hurwitz genus formula applied to $N_1$. Since $G$ and hence $G_1$ is solvable, $G_1$ has a minimal normal $d$-subgroup where $d$ must be equal to $p$ by the choice of $N_1$ to be the largest normal, prime to $p$ subgroup of $G$. Take the largest normal $p$-subgroup $N_2$ of $G_1$. Observe that
$N_2\lneqq G_1$, otherwise $G_1$ is an (elementary) abelian group of order bigger than $34(\gg(\cX_1)+1)^{3/2}$ contradicting the bound $4(\gg(\cX_1)+4)$; see \cite[Theorem 11.79]{HKT}.
Now, define $\cX_2$ to be the quotient curve $\cX_1/N_2$. Since the second ramification group of $N_1$ at any point of $\cX_1$ is trivial, the Hurwitz genus formula together with the Deuring-Shafarevich formula give $\gg(\cX_1)-\gamma(\cX_1)=|N_2|(\gg(\cX_2)-\gamma(\cX_2))$. In particular, $\cX_2$ is neither ordinary or rational by the choice of $\cX$ to be a minimal counterexample. From the proof of Proposition \ref{caseI}, the case $\gg(\cX_2)=1$ cannot occur as $|G_1|>34(\gg(\cX_1)+1)^{3/2}$. Therefore, $\gg(\cX_2)\geq 2$ and $\gg(\cX_2)>\gamma(\cX_2)$ hold. The factor group $G_2=G_1/N_2$ is a $\mathbb{K}$-automorphism group of the quotient curve $\cX_2=\cX_1/N_2$, and it has a minimal normal $d$-subgroup with $d\neq p$, by the choice of $N_2$. Define $N_3$ to be the largest normal, prime to $p$, subgroup of $G_2$. Observe that $N_3$ must be a proper subgroup of $G_2$, otherwise $G_2$ itself would be a prime to $p$ subgroup of $\aut(\cX_2)$ of order bigger than $34(\gg(\cX_2)+1)^{3/2}$ contradicting the classical Hurwitz bound $84(\gg(\cX_2)-1)$. Therefore, there exists a (maximal) nontrivial normal $p$-subgroup $N_4$ in the factor group $G_3=G_2/N_3$.
Now, the above argument remains valid whenever $G,N_1,G_1,N_2,\cX_1,\cX_2$ are replaced by $G_2,N_3,G_3,N_4,\cX_3,\cX_4$ where the quotient curves are $\cX_3=G_2/N_3$ and
$\cX_4=G_3/N_4$. In particular, $\gg(\cX_4)\neq 1$ and $\gg(\cX_3)-\gamma(\cX_3)=|N_4|(\gg(\cX_4)-\gamma(\cX_4))$.  Repeating the above argument a finite sharply decreasing sequence $\gg(\cX_1)>\gg(\cX_2)>\gg(\cX_3)>\gg(\cX_4),\ldots$ arises. If this sequence has $n+1$ members then $\gg(\cX_n)-\gamma(\cX_n)=|N_{n+1}|(\gg(\cX_{n+1})-\gamma(\cX_{n+1}))$ with $\gg(\cX_{n+1})=\gamma(\cX_{n+1})=0$. Therefore, for some (odd) index $m\leq n$, the curve $\cX_m$ for some $m\geq 1$ must be ordinary, a contradiction to choice of $\cX$ to be a minimal counterexample.

 First we investigate the elliptic case. Since $\gg(\cX)\geq 2$, the Hurwitz genus formula applied to $\bar{X}$ ensures that $N$ has a short orbit. Let $\Gamma$ be a short orbit of $G$ containing a short orbit of $N$. Since $N$ is a normal subgroup of $G$, $\Gamma$ is partitioned into short-orbits $\Sigma_1,\ldots,\Sigma_k$ of $N$ each of length $|\Sigma_1|$. Take a point $R_i$ from $\Sigma_i$ for $i=1,2,\ldots, k$, and set $\Sigma=\Sigma_1$ and $S=S_1$. With this notation, $|G|=|G_S||\Gamma|=|G_S| k|\Sigma|$, and the Hurwitz genus formula gives
 \begin{equation}
 \label{eq811oct2016}
 2\gg(\cX)-2 \geq  \sum_{i=1}^k |\Sigma_i|(|N_{S_i}|-1) = k |\Sigma| (|N_{S}|-1|)\geq + \ha k |\Sigma| |N_{S}|=\ha |G| \frac{|N_S|}{|G_S|}.
 \end{equation}
 Also,  the factor group $G_{S}N/N$  is a subgroup of $\aut(\bar{\cX})$ fixing the the point of $\bar{\cX}$ lying under $S$ in the cover $\cX|\bar{\cX}$. From \cite[Theorem 11.94 (ii)]{HKT},
 $$\frac{|G_{S}N|}{|N|} = \frac{|G_{S}|}{|G_{S} \cap N|} = \frac{|G_{S}|}{|N_{S}|}\leq 12.$$
 This and (\ref{eq811oct2016}) yield $|G|\leq 48(\gg(\cX)-1)$, a contradiction with our hypothesis  $34(\gg(\cX)+1)^{3/2}$.

 Therefore, $\bar{\cX}$ is rational. Thus $\bar{G}$ is isomorphic to a subgroup of $PGL(2,\mathbb{K})$. Since $p$ divides $|G|$ but $|N|$, $\bar{G}$ contains a nontrivial $p$-subgroup. From the classification of the finite subgroups of $PGL(2,\mathbb{K})$, see \cite{maddenevalentini1982}, either $p=3$ and $\bar{G}\cong \rm{Alt}_4, \rm{Sym}_4$, or $\bar{G}=\bar{Q}\rtimes \bar{C}$ where $\bar{Q}$ is a normal $p$-subgroup and its complement $\bar{C}$ is a cyclic prime to $p$ subgroup and $|\bar{C}|$ divides $|\bar{Q}|-1$.

 If $\bar{G}\cong \rm{Alt}_4, \rm{Sym}_4$ then $|\bar{G}|\leq 24$ whence $|G|\leq 24|N|\leq 96(\gg(\cX)+1)$ as $N$ is abelian. Comparison with our hypothesis $|G|\geq 34 (\gg(\cX)+1)^{3/2}$ shows that $\gg(\cX)\leq 6$. For small genera we need a little more. If $|N|$ is prime then $|N|\leq 2\gg(\cX)+1$; see \cite[Theorem 11.108]{HKT}, and hence $|G|\leq 48(\gg(\cX)+1)$ which is inconsistent with $|G|\geq 34 (\gg(\cX)+1)^{3/2}$. Otherwise, since $p=3$ and $|N|$ has order a power of prime distinct from $p$, the bound $|N|\leq 4(\gg(\cX)+1)$ with $\gg(\cX)\leq 6$ is only possible for $(\gg(\cX),|N|)\in\{(3,16),(4,16), (5,16),(6,16),(6,25)\}$. Comparison of $|G|\leq 24|N|$ with $|G|\geq 34 (\gg(\cX)+1)^{3/2}$ rule out the latter three cases. Furthermore, since $N$ is an elementary abelian group of order $16$, $\gg(\cX)$ must be odd by Lemma \ref{GK61}. Finally, $\gg(\cX)=3, |N|=16,\, G/N\cong \rm{Sym}_4$ is impossible as Henn's bound $|G|\geq 8\gg(\cX)^3$ implies that $\cX$ has zero $p$-rank, see \cite[Theorem 11.127]{HKT}.

 Therefore, the case $\bar{G}=\bar{Q}\rtimes \bar{C}$ occurs. Also, $\bar{G}$ fixes a unique place $\bar{P} \in \bar{\cX}$.
 Let $\Delta$ be the $N$-orbits in $\cX$ lies over $\bar{P}$ in the cover $\cX|\bar{\cX}$. We prove that $\Delta$ is a long orbit of $N$. By absurd, the permutation representation of $G$ on $\Delta$
 has a nontrivial $1$-point stabilizer containing a nontrivial subgroup $M$ of $N$. Since $N$ is abelian, $M$ is in the kernel. In particular, $M$ is a normal subgroup of $G$ contradicting our choice of $N$ to be minimal.

 Take a Sylow $p$-subgroup $Q$ of $G$ of order $|Q|=p^h$ with $h\geq 1$, and look at the action of $Q$ on $\Delta$. Since $|\Delta|=|N|$ is prime to $p$, $Q$ fixes a point $P\in \Delta$, that is,
 $Q=Q_P$. Since $\cX$ is an ordinary curve, $Q_P$ and hence $Q$ is elementary abelian; see \cite[Theorem 2 (i)]{N} or \cite[Theorem 11.74 (iii)]{HKT}. Therefore, $G_P=Q\rtimes U$ where $U$ is a prime to $p$ cyclic group. Thus
\begin{equation} \label{ord1}
|\bar{Q}||\bar{C}||N|=|\bar{G}||N|=|G|=|G_P||\Delta|=|Q||U||\Delta|=|Q||U||N|,
\end{equation}
whence $|Q|=|\bar{Q}|$ and $|U|=|\bar{C}|$. Consider the subgroup $H$ of $G$ generated by $G_P$ and $N$. Since $\Delta$ is a long $N$-orbit, $G_P \cap N =\{1\}$. As $N$ is normal in $H$ this implies that $H=N \rtimes G_P=N \rtimes (Q \rtimes U)$ and hence $|H|=|N||Q||U|$ which proves $G=H=N \rtimes (Q \rtimes U)$.

Since $\bar{\cX}$ is rational and $\bar{P}$ is the unique fixed point of nontrivial elements of $\bar{Q}$, each $\bar{Q}$-orbit other than $\{\bar{P}\}$ is long. Furthermore, $\bar{C}$ fixes a point
$\bar{R}$ other than $\bar{P}$ and no nontrivial element of $\bar{C}$ is fixes point distinct from $\bar{P}$ and $\bar{R}$. This shows that the $\bar{G}$-orbit $\bar{\Omega}_1$ of $\bar{R}$ has length $|Q|$. In terms of the action of $G$ on $\cX$, there exist as many as $|Q|$ orbits of $N$, say $\Delta_1 ,..., \Delta_{|Q|}$, whose union $\Lambda$ is a short $G$-orbit lying over $\bar{\Omega}_1$ in the cover $\cX|\bar{\cX}$. Obviously, if at least one of $\Delta_i$ is a short $N$-orbit then so all are.

We show that this actually occur. Since the cover $\cX|\bar{\cX}$ ramifies, $N$ has some short orbits, and by absurd there exists a short $N$-orbit $\Sigma$ is not contained in $\Lambda$. Then $\Sigma$ and $\Lambda$ are disjoint. Let $\Gamma$ denote the (short) $G$-orbit containing $\Sigma$. Since $N$ is a normal subgroup of $G$, $\Gamma$ is partitioned into $N$-orbits, say $\Sigma=\Sigma_1,\ldots, \Sigma_k$, each of them of the same length $|\Sigma|$.
Here $k=|Q||U|$ since the set of points of $\bar{\cX}$ lying under these $k$ short $N$-orbits is a long $\bar{G}$-orbit. Also, $|N|=|\Sigma_i||N_{R_i}|$ for $\le i \le k$ and $R_i\in \Sigma_i$. In particular, $|\Sigma_1|=\Sigma_i$ and $|N_{R_1}|=|B_{R_i}$.
From the Hurwitz genus formula,
$$2\gg(\cX)-2 \geq -2|N| + \sum_{i=1}^k |\Sigma_i|(|N_{R_i}|-1) =-2|N| + |Q||U||\Sigma_1|(|N_{R_1}|-1).$$
Since $N_{R_1}$ is nontrivial, $|N_{R_1}| -1 \geq \ha|N_{R_1}|$. Therefore,
$$2\gg(\cX)-2 \geq -2|N| + \ha|Q||U||\Sigma_1||N_{R_1}|=-2|N|+\ha |Q||U||N|=|N|(\ha(|Q||U|-2)=\ha|N|(|Q||U|-4).$$
As $|Q||U|-4 \geq \ha |Q||U|$ by $|Q||U|\geq 4$, this gives
$$2\gg(\cX)-2 \geq \qa |N||U||Q|=\qa |G|.$$
But this contradicts our hypothesis $|G|>34(\gg(\cX)+1)^{3/2}$.

Therefore, the short orbits of $N$ are exactly $\Delta_1 ,..., \Delta_{|Q|}$. Take a point $S_i$ from $\Delta_i$ for $i=1,\ldots,|Q|$. Then $N_{S_1}$ and $N_{S_i}$ are conjugate in $G$, and hence
$|N_{S_1}|=|N_{S_i}|$. From the Hurwitz genus formula applied to $N$,
$$2\gg(\cX)-2=-2|N|+\sum_{i=1}^{|Q|} |\Delta_i|(|N_{S_i}|-1) )=-2|N|+|Q||\Delta_1|(|N_{S_1}|-1)\geq -2|N|+\ha|Q||\Delta_1||N_{S_1}|.$$
Since $|N|= |\Delta_1||N_{S_1}|$, this gives $ 2\gg(\cX)-2 \geq \ha|N|(|Q|-4)|$ whence $2\gg(\cX)-2\geq \qa |N||Q|$
 provided that $|Q|\geq 5$. The missing case, $|Q|=3$, cannot actually occur since in this case $|\bar{C}|=|U|\leq |Q|-1= 2$, whence $|G|=|Q||U||N|\leq 6|N|\leq 24(\gg(\cX)+1)$, a contradiction with $|G|>34(\gg(\cX)+1)^{3/2}$. Thus
\begin{equation}
\label{eq611oct2016}
|N||Q|\le 8(\gg(\cX)-1).
\end{equation}
Since $|N||U|<|N||Q|$, this also shows
\begin{equation}
\label{eq711oct2016}
|N||U|<8(\gg(\cX)-1).
\end{equation}
Therefore, $$|G||N|=|N|^2|U||Q|<64(\gg(\cX)-1)^2.$$
Equations (\ref{eq611oct2016}) and (\ref{eq711oct2016}) together with our hypothesis $|G|\geq 34(\gg(\cX)+1)^{3/2}$ yield
\begin{equation}
\label{eq911oct2016} |N|<\frac{64}{34} \sqrt{\gg(\cX)-1}.
\end{equation}
From (\ref{eq911oct2016}) and $|G|=|N||Q||U|\geq 34(\gg(\cX)+1)^{3/2}$ we obtain
$$|Q||U|>\frac{34^2}{64}(\gg(\cX)-1)> 18(\gg(\cX)-1)$$
which shows that Lemma \ref{terrible} applies to the subgroup $Q\rtimes U$ of $\aut(\cX)$.
With the notation in Lemma \ref{terrible}, this gives that $Q \rtimes U$ and $Q$ have the same two short orbits, $\Omega_1=\{P\}$ and $\Omega_2$. In the cover $\cX/\bar{\cX}$,
the point $\bar{P}\in\bar{\cX}$ lying under $P$ is fixed by $Q$. We prove that $\Omega_2$ is a subset of
the $N$-orbit $\Delta$ containing $P$. For this purpose, it suffices to show that for any point $R\in\Omega_2$, the point $\bar{R}\in \bar{\cX}$ lying under $R$ in the cover $\cX|\bar{\cX}$ coincides  with $\bar{P}$. Since $\Omega_2$ is a $Q$-short orbit, the stabilizer $Q_R$ is nontrivial, and hence $\bar{Q}$ fixes $\bar{R}$. Since $\bar{\cX}$ is rational, this yields $\bar{P}=\bar{R}$.
Therefore, $\Omega_2\cup \{P\}$ is contained in $\Delta$, and either $\Delta=\Omega_2\cup \{P\}$ or $\Delta$ contains a long $Q$-orbit. In the latter case, $|U|<|Q|<|N|$, and hence
$$|G|^2=|N||Q||N|U||Q||U|<|N||Q||N||U||N|^2\le \frac{64^2}{34} (\gg(\cX)-1)^3$$ whence $|G|< 34(\gg(\cX)+1)^{3/2}$, a contradiction with our hypothesis.
Otherwise $|N|=|\Delta|=1+|\Omega_2|$. In particular, $|N|$ is even, and hence it is a power of $2$. Also, from the Deuring-Shafarevi\v{c} Formula,
$\gg(\cX)-1=\gamma(\cX)-1 = -|Q|+1+|\Omega_2|$ where $|\Omega_2|\geq 1$ is a power of $p$. This implies that $\gg(\cX)$ is also even.
Since $N$ is an elementary abelian $2$-group, Lemma \ref{GK61} yields that either $|N|=2$ or $|N|=4$.

If $|N|=2$ then $\Omega_2$ consists of a unique point $R$ and $Q\rtimes U$ fixes both points $P$ and $R$. Since $\Delta=\{P,R\}$, and $\Delta$ is a $G$-orbit, the stabilizer $G_{P,R}$ is an index $2$ (normal) subgroup of $G$. On the other hand, $G_{P,R}=Q\rtimes U$ and hence $Q$ is the unique Sylow $p$-subgroup of $Q\rtimes U$. Thus $Q$ is a characteristic subgroup of the normal subgroup  $G_{P,R}$ of $G$. But then $Q$ is a normal subgroup of $G$, a contradiction with our hypothesis.

 If $|N|=4$ then $|\Delta|=4$ and $p=3$. The permutation representation of $G$ of degree $4$ on $\Delta$ contains a $4$-cycle induced by $N$ but also a $3$-cycle induced by $Q$. Hence if $K=\ker$ then $G/K\cong {\rm{Sym}}_4$. On the other hand, since both $N$ and $\Ker$ are normal subgroups of $G$, their product $NK$ is normal, as well.  Hence $NK/K$ is a normal subgroup of $G/K$, but this contradicts $G/K\cong {\rm{Sym}}_4$.
\end{proof}

\vspace{0.5cm}\noindent {\em Authors' addresses}:

\vspace{0.2cm}\noindent G\'abor KORCHM\'AROS and Maria MONTANUCCI\\ Dipartimento di
Matematica, Informatica ed Economia\\ Universit\`a degli Studi  della Basilicata\\ Contrada Macchia
Romana\\ 85100 Potenza (Italy).\\E--mail: {\tt
gabor.korchmaros@unibas.it and maria.montanucci@unibas.it}

    \end{document}